\documentclass[12pt,a4paper]{article}
\usepackage{mathrsfs}
\usepackage{epsfig, graphicx}
\usepackage{latexsym,amsfonts,amsbsy,amssymb}
\usepackage{amsmath,amsthm}
\makeatletter 

\@addtoreset{equation}{section}
\makeatother 
\allowdisplaybreaks 

\newtheorem{theorem}{Theorem}[section]
\newtheorem{lemma}{Lemma}[section]
\newtheorem{corollary}{Corollary}[section]
\newtheorem{remark}{Remark}[section]
\newtheorem{algorithm}{Algorithm}[section]

\begin{document}

\title{A Cascadic Multigrid Method for Eigenvalue Problem\thanks{This work is supported in part
by the National Science Foundation of China (NSFC 91330202, 11371026, 11001259,
 11031006, 2011CB309703),  the National
Center for Mathematics and Interdisciplinary Science, CAS and the
President Foundation of AMSS-CAS.
}}

\author{Xiaole Han \thanks{LSEC, NCMIS, Institute
of Computational Mathematics, Academy of Mathematics and Systems
Science, Chinese Academy of Sciences, Beijing 100190,
China ({\tt hanxiaole@lsec.cc.ac.cn})}\ \ \ and\ \
Hehu Xie \thanks{LSEC, NCMIS, Institute
of Computational Mathematics, Academy of Mathematics and Systems
Science, Chinese Academy of Sciences, Beijing 100190,
China ({\tt hhxie@lsec.cc.ac.cn})}
}

\date{}
\maketitle

\begin{abstract}
A cascadic multigrid method is proposed for eigenvalue problems based on
the multilevel correction scheme. With this new scheme, an eigenvalue
problem on the finest space
can be solved by smoothing steps on a series of multilevel finite
 element spaces and
 eigenvalue problem solving on the coarsest finite element space.
 Choosing the appropriate sequence of finite element spaces and the number of
 smoothing steps, the optimal convergence rate with the optimal
  computational work can be
 arrived. Some numerical experiments are presented to validate
 our theoretical analysis.


\end{abstract}

\section{Introduction}\label{sec;introduction}
The cascadic multigrid method proposed by \cite{BornemannDeuflhard}
and analyzed by \cite{Shaidurov1996CCG} is based on a hierarchy of
nested meshes.
Going from the lowest level to the highest one, in each level,
the obtained discrete
approximation from the previous level act as the starting value
 of a simple iterative
solver (a smoother) like Conjugate Gradients. As we know,
since a smoother can not reduce the algebraic error from previous
level, this error should already be less than the final desired error.
In cascadic multigrid method, this is achieved by increasing the number
of smoothing iteration steps on lower levels. Fortunately,
the smaller dimensions of the problems
on the lower levels lead to the optimality of the complete algorithm.
Requiring the number of operations which is proportional to the number
of unknowns on the finest level, the algebraic error of the final
approximation solution is of the same order as the discretization error
of the finite element method.

In modern science and engineer, eigenvalue problems appear in many
 fields such as Physics, Chemistry, mechanics and material sciences.
Recently, a type of multilevel correction method is proposed to solve
eigenvalue problems in \cite{LinXie_MultiLevel,Xie_IMA}. In this multilevel
correction scheme, the solution of eigenvalue problem on the final level
mesh can be reduced to a series of solutions of boundary value problems on
the multilevel meshes and a series of solutions of the eigenvalue problem
 on the coarsest mesh. Therefore,  the aim of this paper is to
construct a cascadic multigrid method to solve the eigenvalue problem
by transforming the eigenvalue problem solving to a series of smoothing
 iteration steps on the sequence of meshes and eigenvalue problem solving
  on the coarsest mesh by the multilevel correction method.
Similarly to the cascadic multigrid for the boundary value problem,
we also only do the smoothing steps for a boundary value problem by
using the previous eigenpair approximation
as the start value. As same as the cascadic multigrid method for boundary
value problems, the numbers of smoothing iteration steps need to be
increased in the coarse levels. The final eigenpair approximation has
the same order algebraic error as the discretization error of the
finite element method by organizing the suitable number of smoothing
iteration steps in different levels.

The rest of this paper is organized as follows. In the next section,
 we introduce the finite element method for the eigenvalue problem and
 the corresponding error estimates.
A cascadic multigrid method for eigenvalue problem based on the multilevel
 correction scheme is presented and analyzed in Section 3. In Section 4,
 two numerical examples are presented to validate our theoretical analysis.
  Some concluding remarks are given in the last section.

\section{Finite element method for eigenvalue problem}\label{sec;preliminary}
This section is devoted to introducing some notation and the finite element
method for the eigenvalue problem. In this paper, we shall use the standard notation
for Sobolev spaces $W^{s,p}(\Omega)$ and their
associated norms and semi-norms (\cite{Adams}). For $p=2$, we denote
$H^s(\Omega)=W^{s,2}(\Omega)$ and
$H_0^1(\Omega)=\{v\in H^1(\Omega):\ v|_{\partial\Omega}=0\}$,
where $v|_{\Omega}=0$ is in the sense of trace,
$\|\cdot\|_{s,\Omega}=\|\cdot\|_{s,2,\Omega}$.
In some places, $\|\cdot\|_{s,2,\Omega}$ should be viewed as piecewise
defined if it is necessary.
The letter $C$ (with or without subscripts) denotes a generic
positive constant which may be different at its different occurrences
 through the paper.

For simplicity, we consider the following model problem to illustrate the main idea:
Find $(\lambda, u)$ such that
\begin{equation}\label{LaplaceEigenProblem}
\left\{
\begin{array}{rcl}
-\nabla\cdot(\mathcal{A}\nabla u)&=&\lambda u, \quad {\rm in} \  \Omega,\\
u&=&0, \ \ \ \quad {\rm on}\  \partial\Omega,
\end{array}
\right.
\end{equation}
where $\mathcal{A}$ is a symmetric and positive definite matrix with suitable
regularity, $\Omega\subset\mathcal{R}^d (d=2,3)$ is a bounded domain with
Lipschitz boundary $\partial\Omega$ and $\nabla$, $\nabla\cdot$
denote the gradient, divergence operators, respectively.

In order to use the finite element method to solve
the eigenvalue problem (\ref{LaplaceEigenProblem}), we need to define
the corresponding variational form as follows:
Find $(\lambda, u )\in \mathcal{R}\times V$ such that $b(u,u)=1$ and
\begin{eqnarray}\label{weak_eigenvalue_problem}
a(u,v)&=&\lambda b(u,v),\quad \forall v\in V,
\end{eqnarray}
where $V:=H_0^1(\Omega)$ and
\begin{equation}\label{inner_product_a_b}
a(u,v)=\int_{\Omega}\nabla u\cdot\mathcal{A}\nabla v d\Omega,
 \ \ \ \  \ \ b(u,v) = \int_{\Omega}uv d\Omega.
\end{equation}
The norms $\|\cdot\|_a$ and $\|\cdot\|_b$ are defined by
\begin{eqnarray*}
\|v\|_a=a(v,v)^{1/2}\ \ \ \ \ {\rm and}\ \ \ \ \ \|v\|_b=b(v,v)^{1/2}.
\end{eqnarray*}
  It is well known that the eigenvalue problem (\ref{weak_eigenvalue_problem})
  has an eigenvalue sequence $\{\lambda_j \}$ (cf. \cite{Babuska2,Chatelin}):
$$0<\lambda_1\leq \lambda_2\leq\cdots\leq\lambda_k\leq\cdots,\ \ \
\lim_{k\rightarrow\infty}\lambda_k=\infty,$$ and associated
eigenfunctions
$$u_1, u_2, \cdots, u_k, \cdots,$$
where $b(u_i,u_j)=\delta_{ij}$ ($\delta_{ij}$ denotes the Kronecker function).
In the sequence $\{\lambda_j\}$, the $\lambda_j$ are repeated according to their
geometric multiplicity.

Now, let us define the finite element approximations of the problem
(\ref{weak_eigenvalue_problem}). First we generate a shape-regular
decomposition of the computing domain $\Omega\subset \mathcal{R}^d\
(d=2,3)$ into triangles or rectangles for $d=2$ (tetrahedrons or
hexahedrons for $d=3$). The diameter of a cell $K\in\mathcal{T}_h$
is denoted by $h_K$ and the mesh size $h$ describes  the maximum diameter of all cells
$K\in\mathcal{T}_h$.
Based on the mesh $\mathcal{T}_h$, we can construct a finite element space denoted by
 $V_h \subset V$. For simplicity, we set $V_h$ as the linear finite
 element space which is defined as follows
\begin{equation}\label{linear_fe_space}
  V_h = \big\{ v_h \in C(\Omega)\ \big|\ v_h|_{K} \in \mathcal{P}_1,
  \ \ \forall K \in \mathcal{T}_h\big\},
\end{equation}
where $\mathcal{P}_1$ denotes the linear function space.



The standard finite element scheme for eigenvalue
 problem (\ref{weak_eigenvalue_problem}) is:
Find $(\bar{\lambda}_h, \bar{u}_h)\in \mathcal{R}\times V_h$
such that $b(\bar{u}_h,\bar{u}_h)=1$ and
\begin{eqnarray}\label{Weak_Eigenvalue_Discrete}
a(\bar{u}_h,v_h)
&=&\bar{\lambda}_h b(\bar{u}_h,v_h),\quad\ \  \ \forall v_h\in V_h.
\end{eqnarray}
From \cite{Babuska1,Babuska2}, we have the following
Rayleigh quotient expression for $\bar{\lambda}_h$:
\begin{eqnarray}\label{Rayleigh_Quotient_Discrete}
\bar{\lambda}_h &=&\frac{a(\bar{u}_h,\bar{u}_h)}{b(\bar{u}_h,\bar{u}_h)},
\end{eqnarray}
and the  discrete eigenvalue problem (\ref{Weak_Eigenvalue_Discrete}) has eigenvalues:
$$0<\bar{\lambda}_{1,h}\leq \bar{\lambda}_{2,h}\leq\cdots\leq \bar{\lambda}_{k,h}
\leq\cdots\leq \bar{\lambda}_{N_h,h},$$
and corresponding eigenfunctions
$$\bar{u}_{1,h}, \bar{u}_{2,h},\cdots, \bar{u}_{k,h}, \cdots, \bar{u}_{N_h,h},$$
where $b(\bar{u}_{i,h},\bar{u}_{j,h})=\delta_{ij}, 1\leq i,j\leq N_h$ ($N_h$ is
the dimension of the finite element space $V_h$).

Let $M(\lambda_i)$ denote the eigenspace corresponding to the
eigenvalue $\lambda_i$ which is defined by
\begin{eqnarray}
M(\lambda_i)&=&\big\{w\in H_0^1(\Omega): w\ {\rm is\ an\ eigenvalue\ of\
(\ref{weak_eigenvalue_problem})\ corresponding\ to} \ \lambda_i\nonumber\\
&&\ \ \ \ \ {\rm and}\ b(w,w)=1\big\},
\end{eqnarray}
and define
\begin{eqnarray}
\delta_h(\lambda_i)=\sup_{w\in M(\lambda_i)}\inf_{v\in
V_h}\|w-v\|_{a}.
\end{eqnarray}

Let us define the following quantity:
\begin{eqnarray}
\eta_{a}(h)&=&\sup_{f\in L^2(\Omega),\|f\|_b=1}\inf_{v\in V_h}\|Tf-v\|_{a},\label{eta_a_h_Def}
\end{eqnarray}
where $T:L^2(\Omega)\rightarrow V$ is defined as
\begin{equation}\label{laplace_source_operator}
  a(Tf,v) = b(f,v), \ \ \ \ \  \forall f \in L^2(\Omega) \ \ \  {\rm and}\  \ \ \forall v\in V.
\end{equation}
Then the error estimates for the eigenpair approximations by the finite
element method can be described as follows.
\begin{lemma}(\cite[Lemma 3.7, (3.29b)]{Babuska1}, \cite[P.699]{Babuska2} and \cite{Chatelin})
\label{Err_Eigen_Global_Lem}
For any eigenpair approximation
$(\bar{\lambda}_{i,h},\bar{u}_{i,h})$ $(i = 1, 2, \cdots, N_h)$ of
(\ref{Weak_Eigenvalue_Discrete}),
 there exists an exact eigenpair  $(\lambda_i, u_i)$ of
(\ref{weak_eigenvalue_problem})
such that $b(u_i, u_i) = 1$ and
\begin{eqnarray}
\|u_i-\bar{u}_{i,h}\|_{a}
&\leq& C_i\delta_h(\lambda_i),\label{Err_Eigenfunction_Global_1_Norm} \\
\|u_i-\bar{u}_{i,h}\|_{b}
&\leq& C_i\eta_{a}(h)\|u_i - u_{i,h}\|_{a},\label{Err_Eigenfunction_Global_0_Norm}\\
|\lambda_i-\bar{\lambda}_{i,h}|
&\leq&C_i\delta_h^2(\lambda_i). \label{Estimate_Eigenvalue}
\end{eqnarray}
Here and hereafter $C_i$ is some constant depending on $i$ but independent of  the mesh size $h$.
\end{lemma}

The following Rayleigh quotient expansion of the eigenvalue error is the tool
to obtain the error estimates of the eigenvalue approximations.
\begin{lemma}(\cite{Babuska1})\label{Rayleigh_quotient_expansion_lem}
Assume $(\lambda,u)$ is an eigenpair of the eigenvalue problem
(\ref{weak_eigenvalue_problem}). Then for any $w \in H_0^1(\Omega)\backslash\{0\}$,
the following expansion holds:
\begin{equation}\label{Rayleigh_quotient_expansion}
\frac{a(w,w)}{b(w,w)} - \lambda = \frac{a(w-u,w-u)}{b(w,w)} - \lambda\frac{(w-u,w-u)}{b(w,w)}.
\end{equation}
\end{lemma}

\section{Cascadic multilevel correction scheme for eigenvalue problem}
Recently, a multilevel correction scheme is introduced
in \cite{LinXie_MultiLevel,Xie_IMA} for
solving eigenvalue problems. Here, we propose a type of
cascadic multigrid method for
eigenvalue problems. The main idea in this method is to approximate
 the underlying boundary value
problems on each level by some simple smoothing iteration steps.
In order to describe the cascadic multigrid method, we first introduce
 the sequence of finite element
spaces and the smoothing properties of appropriate smoothers.

In order to do multigrid scheme, we first generate a coarse mesh $\mathcal{T}_H$
with the mesh size $H$ and the coarse linear finite element space $V_H$ is
defined on the mesh $\mathcal{T}_H$. Then we define a sequence of
 triangulations $\mathcal{T}_{h_k}$
of $\Omega\subset \mathcal{R}^d$ determined as follows.
Suppose $\mathcal{T}_{h_1}$ (produced from $\mathcal{T}_H$ by
regular refinements) is given and let $\mathcal{T}_{h_k}$ be obtained
from $\mathcal{T}_{h_{k-1}}$ via regular refinement
(produce $\beta^d$ subelements) such that
\begin{eqnarray}\label{mesh_size_recur}
h_k\approx\frac{1}{\beta}h_{k-1},
\end{eqnarray}
where the positive number $\beta$ denotes the refinement index
 and larger than $1$ (always equals $2$).
Based on this sequence of meshes, we construct the corresponding
 nested linear finite element spaces such that
\begin{eqnarray}\label{FEM_Space_Series}
V_{H}\subseteq V_{h_1}\subset V_{h_2}\subset\cdots\subset V_{h_n}.
\end{eqnarray}
The sequence of finite element spaces
$V_{h_1}\subset V_{h_2}\subset\cdots\subset V_{h_n}$
 and the finite element space $V_H$ have  the following relations
of approximation accuracy
\begin{eqnarray}\label{delta_recur_relation}
\eta_a(H)\gtrsim\delta_{h_1}(\lambda_i),\ \ \ \
\delta_{h_k}(\lambda_i)\approx\frac{1}{\beta}\delta_{h_{k-1}}(\lambda_i),\ \ \ k=2,\cdots,n.
\end{eqnarray}
\begin{remark}
The relation (\ref{delta_recur_relation}) is reasonable since we can choose
$\delta_{h_k}(\lambda_i)=h_k\ (k=1,\cdots,n)$. Always the upper bound of
the estimate $\delta_{h_k}(\lambda_i)\lesssim h_k$ holds. Recently, we also obtain the
lower bound result $\delta_{h_k}(\lambda_i)\gtrsim h_k$ (c.f. \cite{LinXieXu}).
\end{remark}

For generality, we introduce a smoothing operator $S_h: V_h\rightarrow V_h$
which satisfies the following estimate
\begin{equation}\label{smoothing_property}
\left\{
\begin{array}{rcl}
\|S_h^m w_h\|_a &\le& \frac{C}{m^{\alpha}}\frac{1}{h}\|w_h\|_b, \\
\|S_h^m w_h\|_a &\le& \|w_h\|_a, \\
\|S_h^m (w_h + v_h)\|_a &\le&\|S_h^m w_h\|_a + \|S_h^m v_h\|_a,
\end{array}
\right.
\end{equation}
where $C$ is a constant independent of $h$ and $\alpha$
 is some positive number depending on the choice of smoother.
It is proved in \cite{Hackbusch1985,shaidurov1995multigrid,wang2004basic}
that the symmetric Gauss-Seidel,
the SSOR, the damped Jacobi and the Richardson iteration are smoothers in the sense
 of (\ref{smoothing_property}) with parameter $\alpha=1/2$ and the
 conjugate-gradient iteration is the
 smoother with $\alpha=1$ (cf. \cite{Shaidurov1996CCG,ShaidurovTobiska2000}).

Then we define the following notation
\begin{equation}\label{smooth_process_Vh}
w_h = {\it Smooth}(V_h, f, \xi_h, m, S_h)
\end{equation}
as the smoothing process for the following boundary value problem
\begin{equation}\label{fe_source}
a(u_h,v_h) = b(f,v_h), \ \ \ \ \forall v_h \in V_h,
\end{equation}
where $\xi_h$ denote the initial value of the smoothing process, $S_h$ denote the chosen
smoothing operator, $m$ the number of the iteration steps and $w_h$ is
the output of the smoothing process.

Now, we come to introduce the cascadic multigrid method for the eigenvalue problem
(\ref{weak_eigenvalue_problem}). For simplicity, we assume the desired eigenvalue is
simple and the computing domain is convex. Then we have the following estimates
\begin{eqnarray}\label{Estimates_Eta_Delta_Hh}
\eta_a(H)\approx H,\ \ \ \ \eta_a(h_k)\approx h_k
\ \ \ \  {\rm and}\ \ \ \ \delta_{h_k}(\lambda_i)\approx h_k,\ \ \ \ k=1,\cdots,n.
\end{eqnarray}

Assume we have obtained an eigenpair approximations
 $(\lambda^{h_k}, u^{h_k}) \in \mathcal{R}\times V_{h_k}$.
 Now we introduce a cascadic type one correction step to improve the accuracy of the
current eigenpair approximation  $(\lambda^{h_k}, u^{h_k}) \in \mathcal{R}\times V_{h_k}$.
\begin{algorithm}\label{Smooth_Correction_Alg}
Cascadic type of One Correction Step
\begin{enumerate}
\item Define the following auxiliary source problem:
Find $\widehat{u}^{h_{k+1}}\in V_{h_{k+1}}$ such that
\begin{equation}\label{correct_source_exact_cas}
a(\widehat{u}^{h_{k+1}},v_{h_{k+1}}) = \lambda^{h_{k}}b(u^{h_{k}},v_{h_{k+1}}),
\ \ \ \ \ \forall v_{h_{k+1}}\in V_{h_{k+1}}.
\end{equation}
Perform the smoothing process (\ref{smooth_process_Vh}) to obtain a new eigenfuction
 approximation $\widetilde{u}^{h_{k+1}}\in V_{h_{k+1}}$ by
\begin{equation}\label{smooth_correct_process}
\widetilde{u}^{h_{k+1}} = {\it Smooth}(V_{h_{k+1}}, \lambda^{h_{k}}u^{h_{k}},
 u^{h_{k}}, m_{k+1}, S_{h_{k+1}}).
\end{equation}

\item Define a new finite element space $V_{H}^{h_{k+1}} = V_H +
{\rm span}\{\widetilde{u}^{h_{k+1}}\}$ and solve the following eigenvalue problem:
Find $(\lambda^{h_{k+1}},u^{h_{k+1}})\in \mathcal{R}\times V_{H}^{h_{k+1}}$
such that $b(u^{h_{k+1}},u^{h_{k+1}})=1$ and
\begin{equation}\label{cascadic_correct_eig_exact}
a(u^{h_{k+1}},v_{H}^{h_{k+1}}) = \lambda^{h_{k+1}}b(u^{h_{k+1}},v_{H}^{h_{k+1}}),
\ \ \ \ \ \forall v_{H}^{h_{k+1}}\in V_{H}^{h_{k+1}}.
\end{equation}
\end{enumerate}
Summarize the above two steps by defining
\begin{eqnarray*}
(\lambda^{h_{k+1}},u^{h_{k+1}}) =
{\it SmoothCorrection}(V_H,V_{h_{k+1}},\lambda^{h_{k}},u^{h_{k}},m_{k+1},S_{h_{k+1}}).
 \end{eqnarray*}
\end{algorithm}

Based on the above algorithm, i.e., the cascadic type of one correction step, we can
construct a cascadic multigrid method as follows:

\begin{algorithm}\label{Cascadic_MCS_Alg}
Eigenvalue Cascadic Multigrid Method
\begin{enumerate}
\item Find $(\lambda^{h_1}, u^{h_1})\in \mathcal{R}\times V_{h_1}$ such that
\begin{equation*}
a(u^{h_1}, v_{h_1}) = \lambda^{h_1} b(u^{h_1}, v_{h_1}), \quad \forall v_{h_1}\in  V_{h_1}.
\end{equation*}
\item For $k=1,\cdots,n-1$, do the following iteration
\begin{eqnarray*}
(\lambda^{h_{k+1}},u^{h_{k+1}}) =
{\it SmoothCorrection}(V_H,V_{h_{k+1}},\lambda^{h_{k}},u^{h_{k}},m_{k+1},S_{h_{k+1}}).
\end{eqnarray*}
\end{enumerate}
 Finally, we obtain an eigenpair approximation
$(\lambda^{h_n},u^{h_n}) \in \mathcal{R}\times V_{h_n}$.
\end{algorithm}

\vspace{1ex} 
In order to analyze the convergence of Algorithm \ref{Cascadic_MCS_Alg},
 we introduce an auxiliary algorithm and then show its superapproximate property.

Similarly, assume we have obtained an eigenpair approximations
 $(\widetilde{\lambda}_{h_k}, \widetilde{u}_{h_k}) \in \mathcal{R}\times V_{h_k}$.
 We introduce the following auxiliary one correction step.
\begin{algorithm}\label{Auxiliary_Correction_Alg}
Auxiliary One Correction Step
\begin{enumerate}
\item Define the following auxiliary source problem:
Find $\widehat{u}_{h_{k+1}}\in V_{h_{k+1}}$ such that
\begin{equation}\label{Auxiliary_correct_source_exact}
a(\widehat{u}_{h_{k+1}},v_{h_{k+1}}) =
\widetilde{\lambda}_{h_k}b(\widetilde{u}_{h_k},v_{h_{k+1}}),
\ \ \ \ \ \forall v_{h_{k+1}}\in V_{h_{k+1}}.
\end{equation}

\item Define a new finite element space
$\widetilde{V}_{H,h_{k+1}} = V_H +
 {\rm span}\{\widehat{u}_{h_{k+1}}\} + {\rm span}\{\widetilde{u}^{h_{k+1}}\}$
 and solve the following eigenvalue problem:
Find $(\widetilde{\lambda}_{h_{k+1}},\widetilde{u}_{h_{k+1}})\in
 \mathcal{R}\times \widetilde{V}_{H,h_{k+1}}$ such that
  $b(\widetilde{u}_{h_{k+1}},\widetilde{u}_{h_{k+1}})=1$ and
\begin{equation}\label{Auxiliary_correct_eig_exact}
a(\widetilde{u}_{h_{k+1}},\widetilde{v}_{H,h_{k+1}}) = \widetilde{\lambda}_{h_{k+1}}b(\widetilde{u}_{h_{k+1}},\widetilde{v}_{H,h_{k+1}}),
\ \ \ \ \ \forall \widetilde{v}_{H,h_{k+1}}\in \widetilde{V}_{H,h_{k+1}}.
\end{equation}
\end{enumerate}
Summarize the above two steps by defining
\begin{eqnarray*}
(\widetilde{\lambda}_{h_{k+1}},\widetilde{u}_{h_{k+1}}) =
{\it AuxiliaryCorrection}(V_H,V_{h_{k+1}},\widetilde{\lambda}_{h_{k}},\widetilde{u}_{h_{k}},
\widetilde{u}^{h_{k+1}}).
 \end{eqnarray*}
\end{algorithm}
\begin{algorithm}\label{Auxiliary_MCS_Alg}
Eigenvalue Auxiliary Multilevel Correction Method
\begin{enumerate}
\item Find $(\widetilde{\lambda}_{h_1}, \widetilde{u}_{h_1})\in \mathcal{R}\times V_{h_1}$ such that
\begin{equation*}
a(\widetilde{u}_{h_1}, v_{h_1}) = \widetilde{\lambda}_{h_1} b(\widetilde{u}_{h_1}, v_{h_1}),
\quad \forall v_{h_1}\in  V_{h_1}.
\end{equation*}
\item For $k=1,\cdots,n-1$, do the following iteration
\begin{eqnarray*}
(\widetilde{\lambda}_{h_{k+1}},\widetilde{u}_{h_{k+1}}) =
{\it AuxiliaryCorrection}(V_H,V_{h_{k+1}},\widetilde{\lambda}_{h_{k}},\widetilde{u}_{h_{k}},
\widetilde{u}^{h_{k+1}}).
\end{eqnarray*}
\end{enumerate}
Finally, we obtain an eigenpair approximation
$(\widetilde{\lambda}_{h_{n}},\widetilde{u}_{h_{n}}) \in \mathcal{R}\times V_{h_{n}}$.
\end{algorithm}

\vspace{1ex}
Before analyzing the convergence of
 Algorithm \ref{Cascadic_MCS_Alg}, we show a superapproximate
property of $\widetilde{u}_{h_k}$ obtained by
Algorithm \ref{Auxiliary_MCS_Alg}.
\begin{theorem}\label{super_approx_thm}
Assume $\widetilde{u}_{h_k}$ ($k=1,\cdots,n$) are
obtained by Algorithm \ref{Auxiliary_MCS_Alg}
and $\bar{u}_{h_k}$ ($k=1,\cdots,n$) the standard finite
element solution in $V_{h_k}$.  If the sequence of
finite element spaces $V_{h_1}, \cdots,
V_{h_{n}}$ and the coarse finite element space $V_H$
satisfy the following condition
\begin{equation}\label{fe_sequence_condition}
C\eta_a(H)\beta^2 < 1,
\end{equation}
the following estimate holds
\begin{equation}\label{super_approx_eigenfunction}
\|\bar{u}_{h_k} - \widetilde{u}_{h_k}\|_a
\le C\eta_a(h_k)\delta_{h_k}(\lambda),\ \ \ \ \ \ \ k=1,\cdots,n,
\end{equation}
and
\begin{equation}\label{super_approx_eigenfunction_b_norm}
  \|\bar{u}_{h_k} - \widetilde{u}_{h_k}\|_b
  \le C\eta_a(H)\eta_a(h_k)\delta_{h_k}(\lambda),\ \ \ \ \ \ \ k=1,\cdots,n,
\end{equation}
where $C$ is a constant only depending on the eigenvalue $\lambda$.
The eigenvalue approximations $\widetilde{\lambda}_{h_k}$
and $\bar{\lambda}_{h_k}$ have the following estimates
\begin{equation}\label{super_approx_eigenvalue}
\big| \bar{\lambda}_{h_k} - \widetilde{\lambda}_{h_k} \big|
\le \|\bar{u}_{h_k} - \widetilde{u}_{h_k}\|_a^2,\ \ \ \ \ \ \ k=1,2,\cdots,n.
\end{equation}
\end{theorem}
\begin{proof}
  Define $\epsilon_{h_k}:=|\widetilde{\lambda}_{h_k}-\bar{\lambda}_{h_k}|+
  \|\widetilde{u}_{h_k}-\bar{u}_{h_k}\|_b,\ k=1,2,\cdots,n$. And it is obvious that
  $\epsilon_{h_1}=0$.
  From (\ref{Weak_Eigenvalue_Discrete}) and (\ref{Auxiliary_correct_source_exact}), we have
\begin{eqnarray}\label{Error_1}
&& \|\bar{u}_{h_{k+1}} - \widehat{u}_{h_{k+1}}\|_a^2 = a(\bar{u}_{h_{k+1}}
 - \widehat{u}_{h_{k+1}},\bar{u}_{h_{k+1}} - \widehat{u}_{h_{k+1}})\nonumber\\
&=& \bar{\lambda}_{h_{k+1}}b(\bar{u}_{h_{k+1}},\bar{u}_{h_{k+1}}
- \widehat{u}_{h_{k+1}}) -
\widetilde{\lambda}_{h_k}b(\widetilde{u}_{h_k},\bar{u}_{h_{k+1}}
 - \widehat{u}_{h_{k+1}}) \nonumber\\
&=& b(\bar{\lambda}_{h_{k+1}}\bar{u}_{h_{k+1}}
-\widetilde{\lambda}_{h_k}\widetilde{u}_{h_k},
\bar{u}_{h_{k+1}} - \widehat{u}_{h_{k+1}}) \nonumber\\
&\le& \|\bar{\lambda}_{h_{k+1}}\bar{u}_{h_{k+1}}-\widetilde{\lambda}_{h_k}
\widetilde{u}_{h_k}\|_{b} \|\bar{u}_{h_{k+1}} - \widehat{u}_{h_{k+1}}\|_a.
\end{eqnarray}
Note that the eigenvalue problem (\ref{Auxiliary_correct_eig_exact})
 can be regarded
 as a finite dimensional approximation of the eigenvalue problem
  (\ref{Weak_Eigenvalue_Discrete}).
Similarly to Lemma \ref{Err_Eigen_Global_Lem}
(see \cite{Babuska1,LinXie_MultiLevel}),
from the second step in
Algorithm \ref{Auxiliary_Correction_Alg}, the following estimate holds
\begin{equation}\label{Error_2}
\|\bar{u}_{h_{k+1}} - \widetilde{u}_{h_{k+1}} \|_a
\le C\inf_{\widetilde{v}_{H,h_{k+1}}\in\widetilde{V}_{H,h_{k+1}}}
\|\bar{u}_{h_{k+1}}-\widetilde{u}_{H,h_{k+1}}\|_a
\le C \| \bar{u}_{h_{k+1}} - \widehat{u}_{h_{k+1}} \|_a.
\end{equation}
Then combining (\ref{Error_1}) and (\ref{Error_2}) leads to
\begin{eqnarray}\label{super_approx_recur_break_1}
&& \| \bar{u}_{h_{k+1}} - \widetilde{u}_{h_{k+1}} \|_a
\le C \|\bar{\lambda}_{h_{k+1}}\bar{u}_{h_{k+1}}
-\widetilde{\lambda}_{h_k}\widetilde{u}_{h_k}\|_{b}\nonumber\\
&\le& C\big( |\bar{\lambda}_{h_{k+1}} - \widetilde{\lambda}_{h_k}|
+ \|\bar{u}_{h_{k+1}} - \widetilde{u}_{h_{k+1}}\|_{b} \big) \nonumber\\
&\le& C\big( |\bar{\lambda}_{h_{k+1}} - \bar{\lambda}_{h_k}|
+ |\bar{\lambda}_{h_{k}} - \widetilde{\lambda}_{h_k}|
+ \|\bar{u}_{h_{k+1}} - \bar{u}_{h_{k}}\|_{b} +
 \|\bar{u}_{h_{k}} - \widetilde{u}_{h_{k}}\|_{b}\big) \nonumber\\
&\le& C\big( |\bar{\lambda}_{h_{k+1}} - \bar{\lambda}_{h_k}|
+ \|\bar{u}_{h_{k+1}} - \bar{u}_{h_{k}}\|_{b} + \epsilon_{h_k}\big).
\end{eqnarray}
From the properties of $V_{h_k}\subset V_{h_{k+1}}$, $V_{H,h_k}\subset V_{h_k}$,
Lemma \ref{Err_Eigen_Global_Lem} and (\ref{delta_recur_relation}), we have
\begin{eqnarray*}
&&\|\bar{u}_{h_{k+1}} - \bar{u}_{h_{k}}\|_a\le
C\delta_{h_{k}}(\lambda), \ \ \ \ \
\|\bar{u}_{h_{k+1}} - \bar{u}_{h_{k}}\|_b \le
 C\eta_a(h_{k})\|\bar{u}_{h_{k+1}} - \bar{u}_{h_{k}}\|_a,\nonumber\\
&&\big| \bar{\lambda}_{h_{k+}} - \bar{\lambda}_{h_{k}} \big|
\leq C\|\bar{u}_{h_{k+1}} - \bar{u}_{h_{k}}\|_a^2
\leq C\delta_{h_{k}}(\lambda)^2\leq C\eta_a(h_k)\delta_{h_k}(\lambda) \\
&&\| \bar{u}_{h_k} - \widetilde{u}_{h_k} \|_b
\le C\eta_a(H)\| \bar{u}_{h_k} - \widetilde{u}_{h_k} \|_{a}, \ \ \ \ \
\big| \bar{\lambda}_{h_{k}} - \widetilde{\lambda}_{h_{k}} \big|
\le C\| \bar{u}_{h_k} - \widetilde{u}_{h_k} \|_{a}^2.
\end{eqnarray*}
Substituting above inequalities into (\ref{super_approx_recur_break_1}) leads to the following estimates
\begin{eqnarray}\label{super_approx_recur_break_2}
\| \bar{u}_{h_{k+1}} - \widetilde{u}_{h_{k+1}} \|_a &\le& C\big( \delta^2_{h_k}(\lambda)
+ \eta_a(h_k)\delta_{h_k}(\lambda) + \epsilon_{h_k} \big) \nonumber\\
&\le& C\big(\eta_a(h_k)\delta_{h_k}(\lambda) +
 \eta_a(H)\|\bar{u}_{h_k} - \widetilde{u}_{h_k}\|_{a} \big).
\end{eqnarray}
When $k=1$, since $\widetilde{u}_{h_1}:=\bar{u}_{h_1}$
 and $\widetilde{\lambda}_{h_1}:=\bar{\lambda}_{h_1}$, we have
\begin{eqnarray}\label{Error_k_1}
\| \bar{u}_{h_{2}} - \widetilde{u}_{h_{2}} \|_a \leq C\eta_a(h_1)\delta_{h_1}(\lambda).
\end{eqnarray}
Based on (\ref{delta_recur_relation}), (\ref{super_approx_recur_break_2}), (\ref{Error_k_1})
and recursive argument, we have the following estimates:
\begin{eqnarray}\label{super_approx_recur_break_3}
\|\bar{u}_{h_{k}} - \widetilde{u}_{h_{k}}\|_a &\le& C\sum_{j=2}^k C^{k-j}\eta_a^{k-j}(H)
\eta_a(h_{j-1})\delta_{h_{j-1}}(\lambda)  \nonumber \\
&\le& C\sum_{j=2}^k C^{k-j}\eta_a^{k-j}(H)\beta^{k-j+1}
\eta_a(h_{k})\beta^{k-j+1}\delta_{h_k}(\lambda) \nonumber\\
&\le& C{\beta}^2\Big(\sum_{j=2}^k \big(C\eta_a(H)\beta^2\big)^{k-j}
\Big)\eta_a(h_{k})\delta_{h_k}(\lambda)\nonumber\\
&\le& \frac{C\beta^2}{1-C\beta^2\eta_a(H)}\eta_a(h_{k})\delta_{h_k}(\lambda).
\end{eqnarray}
Therefore,  the desired result (\ref{super_approx_eigenfunction}) holds under
the condition $C\eta_a(H)\beta^2<1$. Furthermore, (\ref{super_approx_eigenfunction_b_norm})
and (\ref{super_approx_eigenvalue}) can be obtained directly from Lemmas \ref{Err_Eigen_Global_Lem}
and \ref{Rayleigh_quotient_expansion_lem}, respectively.
\end{proof}


\vspace{1ex}
Note that $V_H^{h_k}\subset\widetilde{V}_{H,h_k}$, then we can obtain
the following estimates which play an important role in our analysis.

\begin{lemma}\cite[Lemma 3.5]{Babuska1}\label{E_h_less_P_h_Lem}
Let $u^{h_k}$, $V_H^{h_k}$ and $\widetilde{u}_{h_k}$, $\widetilde{V}_{H,h_k}$ be
defined in Algorithms \ref{Smooth_Correction_Alg} and \ref{Auxiliary_Correction_Alg}.
Then the following estimates hold:
\begin{eqnarray}
\|u^{h_k} - \widetilde{u}_{h_k}\|_a &\le&
 C\|\widehat{u}_{h_k}-\widetilde{u}^{h_k}\|_a, \label{E_h_less_P_h_a}\\
\|u^{h_k} - \widetilde{u}_{h_k}\|_b &\le&
C\eta_a(H)\|u^{h_k} - \widetilde{u}_{h_k}\|_a, \label{E_h_less_P_h_b}\\
|\lambda^{h_k} - \widetilde{\lambda}_{h_k}| &\le&
 \|u^{h_k} - \widetilde{u}_{h_k}\|_a^2. \label{E_h_less_P_h_eigenvalue}
\end{eqnarray}
\end{lemma}
\begin{proof}
Since $V_H^{h_k}\subset\widetilde{V}_{H,h_k}$, according to (\ref{cascadic_correct_eig_exact})
 and (\ref{Auxiliary_correct_eig_exact}), $u^{h_k}$
  can be viewed as the spectral projection of $\widetilde{u}_{h_k}$ (cf. \cite{Babuska1}).
  Then from Lemma \ref{Err_Eigen_Global_Lem} and the definitions of $\widetilde{V}_{H,h_k}$
  and $V_H^{h_k}$, we have
\begin{eqnarray}
\|\widetilde{u}_{h_k}-u^{h_k}\|_a &\le& C\inf_{v_H^{h_k}
\in V_H^{h_k}}\|\widetilde{u}_{h_k}-v_H^{h_k}\|_a \le
 C\inf_{v_H^{h_k}\in V_H^{h_k}}\|\widehat{u}_{h_k}-v_H^{h_k}\|_a \nonumber\\
&\le& C\|\widehat{u}_{h_k}-\widetilde{u}^{h_k}\|_a,
\end{eqnarray}
which is the desired result (\ref{E_h_less_P_h_a}).
Similarly, we also have (\ref{E_h_less_P_h_b}) by the following argument
\begin{eqnarray*}
\|\widetilde{u}_{h_{k}} - u^{h_{k}}\|_b
\le C\eta_a(V_H^{h_k})\|\widetilde{u}_{h_{k}} - u^{h_{k}}\|_a
\le C\eta_a(H)\|\widetilde{u}_{h_{k}} - u^{h_{k}}\|_a,
\end{eqnarray*}
where
\begin{equation*}
\eta_a(V_H^{h_k}):=\sup_{f\in L^2(\Omega),\|f\|_b=1}\inf_{v\in V_H^{h_k}}\|Tf-v\|_{a}\leq \eta_a(H).
\end{equation*}
Furthermore, (\ref{E_h_less_P_h_eigenvalue}) can be obtained directly
from Lemma \ref{Rayleigh_quotient_expansion_lem}
and the proof is complete.
\end{proof}

\begin{remark}
Since $V_H\subset V_H^{h_k}$ and $V_H\subset \widetilde{V}_{H,h_k}$,
from Lemma \ref{Err_Eigen_Global_Lem}, we have
\begin{equation}\label{u_hat_u_tilde_a_less_H}
\|u^{h_{k}} - \widetilde{u}_{h_{k}}\|_a \le \|u^{h_{k}} - u\|_a +
\| u- \widetilde{u}_{h_{k}}\|_a \le C\delta_{H}(\lambda).
\end{equation}
\end{remark}

\vspace{1ex} 
Now, we come to give error estimates for Algorithm \ref{Cascadic_MCS_Alg}.
\begin{theorem}\label{Cascadic_Convergence_Thm}
Assume the eigenpair approximation $(\lambda^{h_n}, u^{h_n})$ is
obtained by Algorithm \ref{Cascadic_MCS_Alg}, $(\widetilde{\lambda}_{h_n}, \widetilde{u}_{h_n})$
is obtained by Algorithm \ref{Auxiliary_MCS_Alg} and the smoother selected in each
level $V_{h_k}$ satisfy the smoothing property (\ref{smoothing_property}) for $k=1, \cdots, n$.
Under the conditions of Theorem \ref{super_approx_thm},
we have the following estimate:
\begin{equation}\label{cascadic_convergence_eigenfunction}
\|\widetilde{u}_{h_n} - u^{h_n}\|_a
\le C\sum_{k=2}^n \frac{\big(1+C\eta_a(H)\big)^{n-k}}{m_k^{\alpha}}\delta_{h_k}(\lambda),
\end{equation}
and the corresponding eigenvalue error estimate
\begin{equation}\label{cascadic_convergence_eigenvalue}
\big| \widetilde{\lambda}_{h_n} - \lambda^{h_n} \big|
\le C\|\widetilde{u}_{h_n} - u^{h_n}\|_a^2.
\end{equation}
\end{theorem}
\begin{proof}
 Define $e_{h_k}:=u^{h_k} - \widetilde{u}_{h_k}$ for $k=1, \cdots, n$.
 Then it is easy to see that $e_{h_1}=0$.

 From Lemma \ref{E_h_less_P_h_Lem}, the following inequalities hold
\begin{eqnarray} \label{cascadic_convergence_break_1}
\|e_{h_{k+1}}\|_a &=& \|u^{h_{k+1}} - \widetilde{u}_{h_{k+1}}\|_a
\le C\|\widehat{u}_{h_{k+1}} - \widetilde{u}^{h_{k+1}}\|_a \nonumber\\
&\le& C\Big(\|\widehat{u}_{h_{k+1}} - \widehat{u}^{h_{k+1}}\|_a
+ \|\widehat{u}^{h_{k+1}} - \widetilde{u}^{h_{k+1}}\|_a\Big).
\end{eqnarray}
For the first term in (\ref{cascadic_convergence_break_1}), together
 with (\ref{correct_source_exact_cas}), (\ref{Auxiliary_correct_source_exact}),
 Lemma \ref{E_h_less_P_h_Lem} and (\ref{u_hat_u_tilde_a_less_H}), we have
\begin{eqnarray}\label{cascadic_convergence_break_2}
\|\widehat{u}_{h_{k+1}} - \widehat{u}^{h_{k+1}}\|_a &\le& C\|\lambda^{h_k}u^{h_k}
- \widetilde{\lambda}_{h_k}\widetilde{u}_{h_k}\|_b \le C\Big( \|u^{h_k}
-\widetilde{u}_{h_k}\|_a^2 + \|u^{h_k}-\widetilde{u}_{h_k}\|_b \Big) \nonumber\\
&\le& C\eta_a(H)\|u^{h_k}-\widetilde{u}_{h_k}\|_a = C\eta_a(H)\|e_{h_k}\|_a.
\end{eqnarray}
For the second term in (\ref{cascadic_convergence_break_1}), due to
(\ref{smoothing_property}) and (\ref{cascadic_convergence_break_2}),
 the following estimates hold
\begin{eqnarray}
&&\|\widehat{u}^{h_{k+1}} - \widetilde{u}^{h_{k+1}}\|_a  =
\|S_{h_{k+1}}^{m_{k+1}}(\widehat{u}^{h_{k+1}} - u^{h_k})\|_a   \nonumber \\
&\le& \|S_{h_{k+1}}^{m_{k+1}}(\widehat{u}^{h_{k+1}} - \widetilde{u}_{h_k})\|_a
+ \|S_{h_{k+1}}^{m_{k+1}}(\widetilde{u}_{h_k} - u^{h_k})\|_a  \nonumber \\
&\le& \|S_{h_{k+1}}^{m_{k+1}}(\widehat{u}^{h_{k+1}}-\widehat{u}_{h_{k+1}})\|_a
+ \|S_{h_{k+1}}^{m_{k+1}}(\widehat{u}_{h_{k+1}} - \widetilde{u}_{h_{k}})\|_a
+ \|\widetilde{u}_{h_k} - u^{h_k}\|_a\nonumber \\
&\le& \|\widehat{u}_{h_{k+1}} - \widehat{u}^{h_{k+1}}\|_a +
\frac{C}{m_{k+1}^{\alpha}}\frac{1}{h_{k+1}}\|\widehat{u}_{h_{k+1}}
- \widetilde{u}_{h_{k}}\|_b + \|\widetilde{u}_{h_k} - u^{h_k}\|_a \nonumber \\
&\le&(1+C\eta_a(H))\|e_{h_k}\|_a +
\frac{C}{m_{k+1}^{\alpha}}\frac{1}{h_{k+1}}\|\widehat{u}_{h_{k+1}}
- \widetilde{u}_{h_{k}}\|_b. \label{cascadic_convergence_break_3}
\end{eqnarray}
According to Lemma \ref{Err_Eigen_Global_Lem}, (\ref{delta_recur_relation}),
Theorem \ref{super_approx_thm} and its proof,
\begin{eqnarray} \label{cascadic_convergence_break_4}
\|\widehat{u}_{h_{k+1}} - \widetilde{u}_{h_{k}}\|_b &\le&
\|\widehat{u}_{h_{k+1}} - \bar{u}_{h_{k+1}}\|_b +
\|\bar{u}_{h_{k+1}} - \bar{u}_{h_{k}}\|_b + \|\bar{u}_{h_{k}} - \widetilde{u}_{h_{k}}\|_b \nonumber\\
&\le& C\eta_a(h_{k+1})\delta_{h_{k+1}}(\lambda).
\end{eqnarray}
Combining (\ref{cascadic_convergence_break_1}), (\ref{cascadic_convergence_break_2}),
(\ref{cascadic_convergence_break_3}), (\ref{cascadic_convergence_break_4})
 and (\ref{Estimates_Eta_Delta_Hh}), we have
\begin{equation}\label{e_k_recur}
\|e_{h_{k+1}}\|_a \le \big(1+C\eta_a(H)\big)\|e_{h_k}\|_a
+ \frac{C}{m_{k+1}^{\alpha}}\delta_{h_{k+1}}(\lambda), \ \ k=1,\cdots,n-1.
\end{equation}
Based on (\ref{e_k_recur}), the fact $e_{h_1}=0$ and the
recursive argument, the following estimates hold
\begin{eqnarray*}
\|e_{h_n}\| &\le& \big(1+C\eta_a(H)\big)\|e_{h_{n-1}}\|_a +
\frac{C}{m_{n}^{\alpha}}\delta_{h_{n}}(\lambda)  \\
&\le& \big(1+C\eta_a(H)\big)^2\|e_{h_{n-2}}\|_a +
\big(1+C\eta_a(H)\big)\frac{C}{m_{n-1}^{\alpha}}\delta_{h_{n-1}}(\lambda)
+\frac{C}{m_{n}^{\alpha}}\delta_{h_{n}}(\lambda) \\
&\le& C\sum_{k=2}^n \big(1+C\eta_a(H)\big)^{n-k}\frac{1}{m_k^{\alpha}}\delta_{h_k}(\lambda).
\end{eqnarray*}
This is the desired result (\ref{cascadic_convergence_eigenfunction}).
The estimate (\ref{cascadic_convergence_eigenvalue}) can be obtained
from Lemma \ref{Rayleigh_quotient_expansion_lem} and
(\ref{cascadic_convergence_eigenfunction}).
\end{proof}

\begin{corollary}
Under the conditions of Theorem \ref{Cascadic_Convergence_Thm},
we have the following estimates:
\begin{eqnarray}
\|\bar{u}_{h_n} - u^{h_n}\|_a &\le& C\Big( \eta_a(h_n)\delta_{h_n}(\lambda)+  \sum_{k=2}^n \frac{\big(1+C\eta_a(H)\big)^{n-k}}{m_k^{\alpha}}\delta_{h_k}(\lambda)  \Big),\\
|\bar{\lambda}_{h_n} - \lambda^{h_n}| &\le& \|\bar{u}_{h_n} - u^{h_n}\|_a^2.
\end{eqnarray}
\end{corollary}

\vspace{1ex}
Now we come to estimate the computational work for Algorithm \ref{Cascadic_MCS_Alg}.
Define the dimension of each linear finite element space as
\begin{equation*}
N_k:=\text{\rm dim}~V_{h_k},\ \ k=1,\cdots,n.
\end{equation*}
Then we have
\begin{equation}\label{N_h_recur_relation}
N_k\approx\Big(\frac{h_k}{h_n}\Big)^{-d}N_n = \Big(\frac{1}{\beta}\Big)^{d(n-k)}N_n,\ \ k=1,\cdots,n.
\end{equation}
From Theorem \ref{Cascadic_Convergence_Thm}, in order to control the global error,
 it is required that the number of iterations in the coarser spaces
 should be larger than the fine spaces. To give a precise analysis for the final
 error and complexity estimates, we assume the following inequality holds for the
  number of iterations in each level mesh:
\begin{equation}\label{iter_reccur_relation}
\Big(\frac{h_k}{h_{n}}\Big)^{\zeta} \le \frac{m_k^{\alpha}}{\bar{m}^{\alpha}}
 \le \sigma \Big(\frac{h_k}{h_{n}}\Big)^{\zeta},\ \ \ \ \ \ \  k=2,\cdots,n-1,
\end{equation}
where $ \bar{m} = m_{n} $, $\sigma > 1$ and $\zeta >1$ are some appropriate constants.

Now, we give the final error and the complexity estimates for Algorithm \ref{Cascadic_MCS_Alg}.
\begin{theorem}\label{estimate_number_iter}
Under the conditions  (\ref{delta_recur_relation}), (\ref{iter_reccur_relation})
and $\beta^{1-\zeta}(1+CH)<1$, for any given $\gamma \in (0,1]$, the final error estimate
\begin{eqnarray}\label{Final_Error_Estimate}
\|u^{h_{n}} - \widetilde{u}_{h_{n}}\|_a \le \gamma h_{n}
\end{eqnarray}
holds if we take
\begin{eqnarray}\label{Condition_m_bar}
\bar{m}>\Big(\frac{CC_{\zeta}}{\gamma}\Big)^{\frac{1}{\alpha}},
\end{eqnarray}
where $C_\zeta = {1}/{(1-\beta^{1-\zeta}(1+CH))}$.

Assume the eigenvalue problem solved in the coarse spaces $V_{H}$ and $V_{h_1}$ need work
$M_H$ and $M_{h_1}$, respectively.
If $\zeta/\alpha <d$, the total computational work of Algorithm \ref{Cascadic_MCS_Alg} can be bounded by
$\mathcal{O}(N_{n}+M_{h_1}+M_H\log(N_{n}))$  and furthermore $\mathcal{O}(N_{n})$
provided  $M_H \ll N_{n}$ and $M_{h_1}\leq N_{n}$.
while if $\zeta/\alpha =d$, the total computational work can be bounded by
$\mathcal{O}(N_{n}\log(N_{n})+M_{h_1}+M_H\log(N_{n}))$  and furthermore $\mathcal{O}(N_{n}\log(N_{n}))$
provided  $M_H \ll N_{n}$ and $M_{h_1}\leq N_{n}$.
\end{theorem}
\begin{proof}
By Theorem \ref{Cascadic_Convergence_Thm}, together with (\ref{mesh_size_recur}),
 (\ref{Estimates_Eta_Delta_Hh}), (\ref{cascadic_convergence_eigenfunction})
 and (\ref{iter_reccur_relation}), we have the following estimates:
\begin{eqnarray}\label{m_k_cascadic_convergence}
\|u^{h_{n}} - \widetilde{u}_{h_{n}}\|_a &\leq& C \sum_{k=2}^{n}(1+C\eta_a(H))^{n-k}
\frac{1}{m_k^{\alpha}}\delta_{h_k}(\lambda) \le C \sum_{k=2}^{n}
 (1+CH)^{n -k} \frac{1}{\bar{m}^{\alpha}}
 \Big(\frac{h_k}{h_{n}}\Big)^{-\zeta}h_k \nonumber\\
&\le& C\sum_{k=2}^{n} (1+CH)^{n -k}
\beta^{(n-k)(1-\zeta)}\frac{h_{n}}{\bar{m}^{\alpha}}
= C\frac{h_{n}}{\bar{m}^{\alpha}}\sum_{k=0}^{n-2}
 \big(\beta^{1-\zeta}(1+CH)\big)^{k}\nonumber\\
&=& C\frac{h_{n}}{\bar{m}^{\alpha}}\frac{1}{1-\beta^{1-\zeta}(1+CH)}.
\end{eqnarray}
When $\beta^{1-\zeta}(1+CH)<1$, (\ref{m_k_cascadic_convergence}) becomes
\begin{equation}\label{m_l_cascadic_convergence}
\|u^{h_{n}} - \bar{u}_{h_{n}}\|_a \le \frac{CC_{\zeta}}{\bar{m}^{\alpha}}h_{n}.
\end{equation}
 Then it is obvious that we can obtain
$\|u^{h_{n}} - \widetilde{u}_{h_{n}}\|_a \le \gamma h_{n}$ when $\bar{m}$ satisfies the condition
(\ref{Condition_m_bar}).

Let $W$ denote the whole computational work of Algorithm \ref{Cascadic_MCS_Alg},
$w_k$ the work on the $k$-th level for $k=1,\cdots,n$.
Based on the definition of Algorithms \ref{Smooth_Correction_Alg} and \ref{Cascadic_MCS_Alg},
(\ref{mesh_size_recur}), (\ref{iter_reccur_relation})  and (\ref{N_h_recur_relation}), the following estimates hold
\begin{eqnarray*}
W &=& \sum_{k=1}^{n} w_k \le M_{h_1} + \sum_{k=2}^{n} m_k N_k + M_H \log_{\beta}(N_{n}) \\
&\leq& M_{h_1} + C M_H\log(N_{n})+ \bar{m}\sigma^{1/\alpha}
N_{n}\sum_{k=2}^{n} \Big(\frac{1}{\beta}\Big)^{(n-k)(d-\zeta/\alpha)}.
\end{eqnarray*}
Then we know that the computation work $W$ can be bounded by $\mathcal{O}(M_{h_1}+M_H\log(N_{n})+N_{n})$
when $d-\zeta/\alpha >0$ and by  $\mathcal{O}(M_{h_1}+M_H\log(N_{n})+N_{n}\log(N_{n}))$ when
 $d-\zeta/\alpha =0$.  It is also obvious they can be bounded by $\mathcal{O}(N_{n})$ and
 $\mathcal{O}(N_{n}\log(N_{n}))$, respectively, if $M_H \ll N_{n}$ and $M_{h_1}\le N_{n}$ are provided.
\end{proof}

\begin{corollary}
  Under the same conditions of Theorem \ref{estimate_number_iter} and (\ref{Condition_m_bar}) holding, if $Ch_n\le\gamma$, then we have the following estimate
  \begin{equation}\label{final_error_u_bar}
    \|u^{h_n}-\bar{u}_{h_n}\|_a \le 2\gamma h_n.
  \end{equation}
\end{corollary}

\vspace{1ex}
If we choose the conjugate gradient method as the smoothing operator, then $\alpha = 1$
and the computation work of Algorithm \ref{Cascadic_MCS_Alg} can be bounded by
$\mathcal{O}(N_{n}+M_{h_1}+M_H\log(N_{n}))$ or $\mathcal{O}(N_{n})$
provided $M_H \ll N_{n}$ and $M_{h_1}\leq N_{n}$  for both
 $d=2$ and $d=3$ when we choose $1 < \zeta < d$.

When the symmetric Gauss-Seidel, the SSOR, the damped Jacobi or the Richardson  iteration act as
the smoothing  operator,  we know $\alpha=1/2$. Then  the computation work of
Algorithm \ref{Cascadic_MCS_Alg} can be bounded by
 $\mathcal{O}(N_{n}+M_{h_1}+M_H\log(N_{n}))$ ($\mathcal{O}(N_{n})$
provided $M_H \ll N_{n}$ and $M_{h_1}\leq N_{n}$)  only for $d=3$  when we choose $1 < \zeta < 3/2$.
In the case  of $\alpha=1/2$  and $d=2$,  from Theorem \ref{estimate_number_iter} and its proof, we can
only choose $\zeta=1$ and then the final error has the estimate
$\|u^{h_{n}}-\bar{u}_{h_{n}}\|_a \leq Ch_{n}|\log(h_{n})|$
and the computational work can only be bounded by
$\mathcal{O}(N_{n}\log(N_{n})+M_{h_1}+M_H\log(N_{n}))$ ($\mathcal{O}(N_{n}\log(N_{n}))$
provided $M_H \ll N_{n}$ and $M_{h_1}\leq N_{n}$).

\section{Numerical tests}
In this section, two numerical examples are presented to illustrate the
efficiency of the cascadic multigrid scheme (Algorithm \ref{Cascadic_MCS_Alg})
 proposed in this paper. Here, we choose the conjugate-gradient iteration as the
smoothing operator ($\alpha=1$) and the number of iteration steps by
\begin{eqnarray*}
m_k = \lceil\sigma\times 2^{\zeta(n-k)}\rceil\ \ \ {\rm for}\  k=2, \cdots, n
\end{eqnarray*}
with $\sigma=2$, $\zeta=1.01$ and $\lceil r\rceil$
denoting the smallest integer which is not less than $r$
\subsection{Model eigenvalue problem}
Here we give the numerical results of the cascadic multigrid
scheme for Laplace eigenvalue problem on the two dimensional domain
$\Omega=(0,1 )\times (0, 1)$.  The sequence of
finite element spaces are constructed by
using linear element on the series of mesh which are produced by
regular refinement with $\beta =2$ (connecting the midpoints of each edge).
In this example, we use two meshes which are generated by Delaunay method as the initial mesh
$\mathcal{T}_{h_1}$ and set $\mathcal{T}_H=\mathcal{T}_{h_1}$
to investigate the convergence behaviors.
Figure \ref{Initial_Mesh} shows the corresponding
initial meshes: one is coarse and the other is fine.

Algorithm \ref{Cascadic_MCS_Alg} is applied to solve the eigenvalue problem.
For comparison, we also solve the eigenvalue problem by the direct finite element method.
\begin{figure}[htb]
\centering
\includegraphics[width=5cm,height=4.5cm]{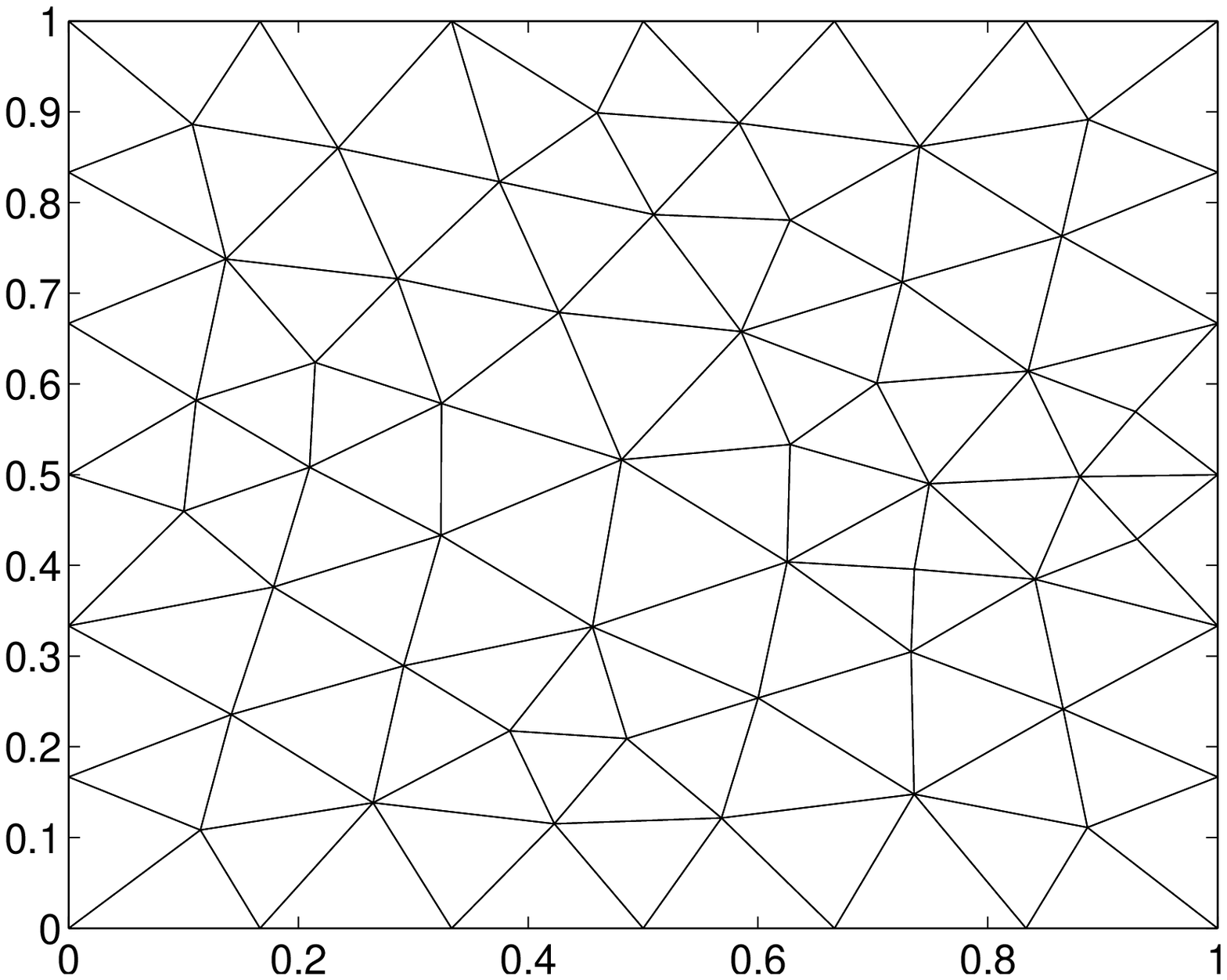}
\includegraphics[width=5cm,height=4.5cm]{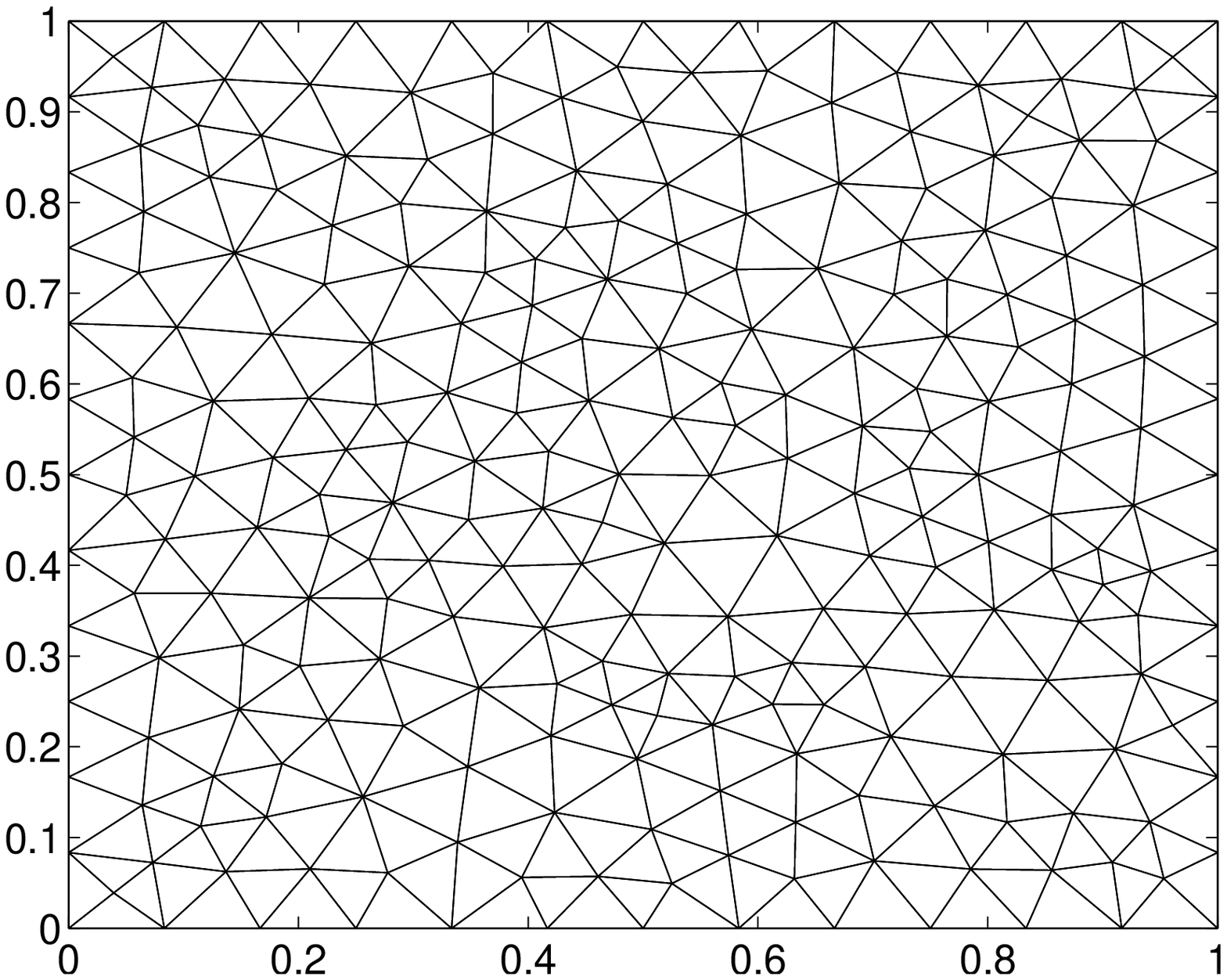}
\caption{\small\texttt The coarse and fine initial meshes for Example 1}
\label{Initial_Mesh}
\end{figure}

Figure \ref{numerical_multi_grid_2D}
gives the corresponding numerical results for the first eigenvalue
$\lambda_1=2\pi^2$ and the corresponding eigenfunction on the two initial meshes
 illustrated in Figure \ref{Initial_Mesh}.
\begin{figure}[htb]
\centering
\includegraphics[width=6cm,height=5cm]{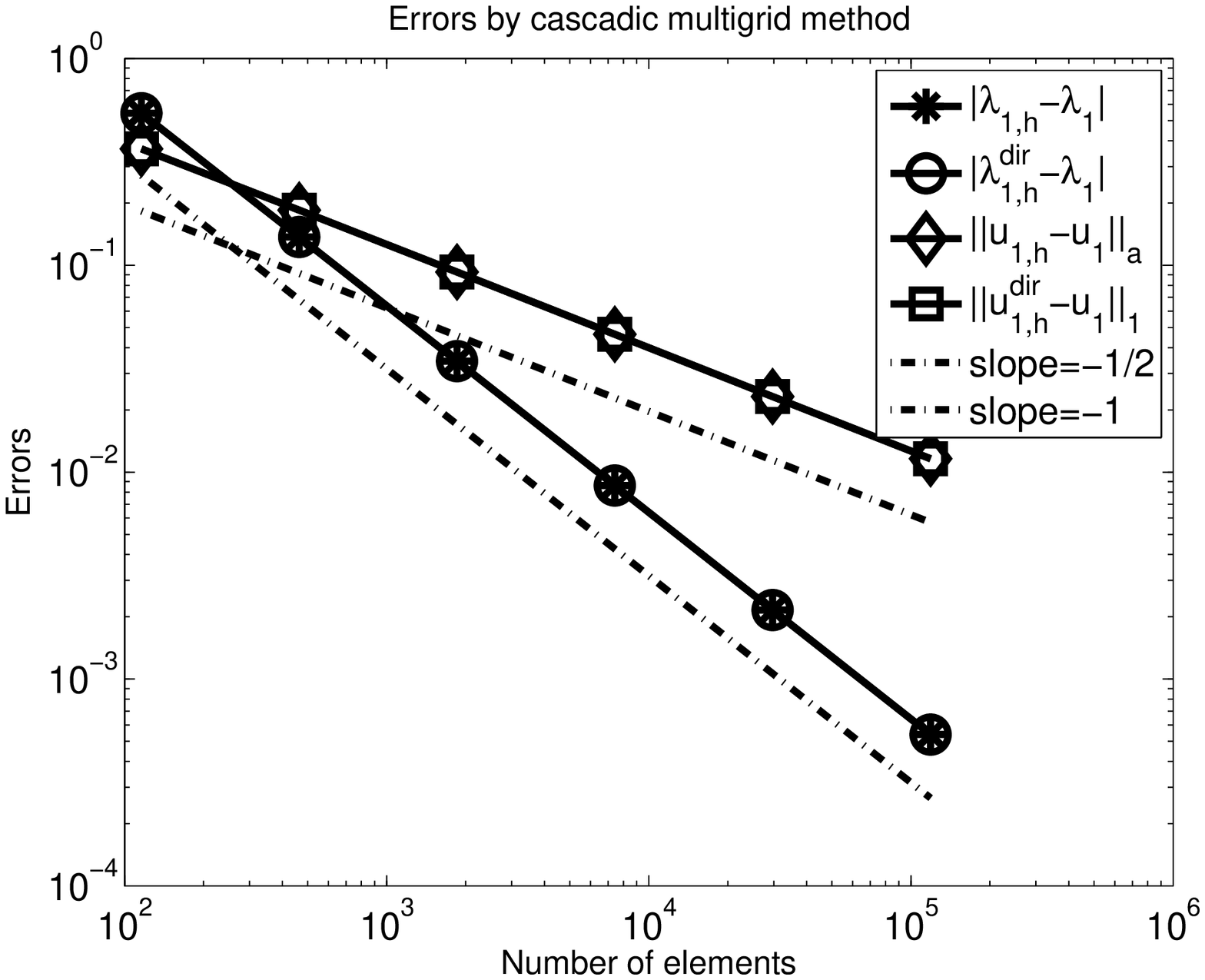}
\includegraphics[width=6cm,height=5cm]{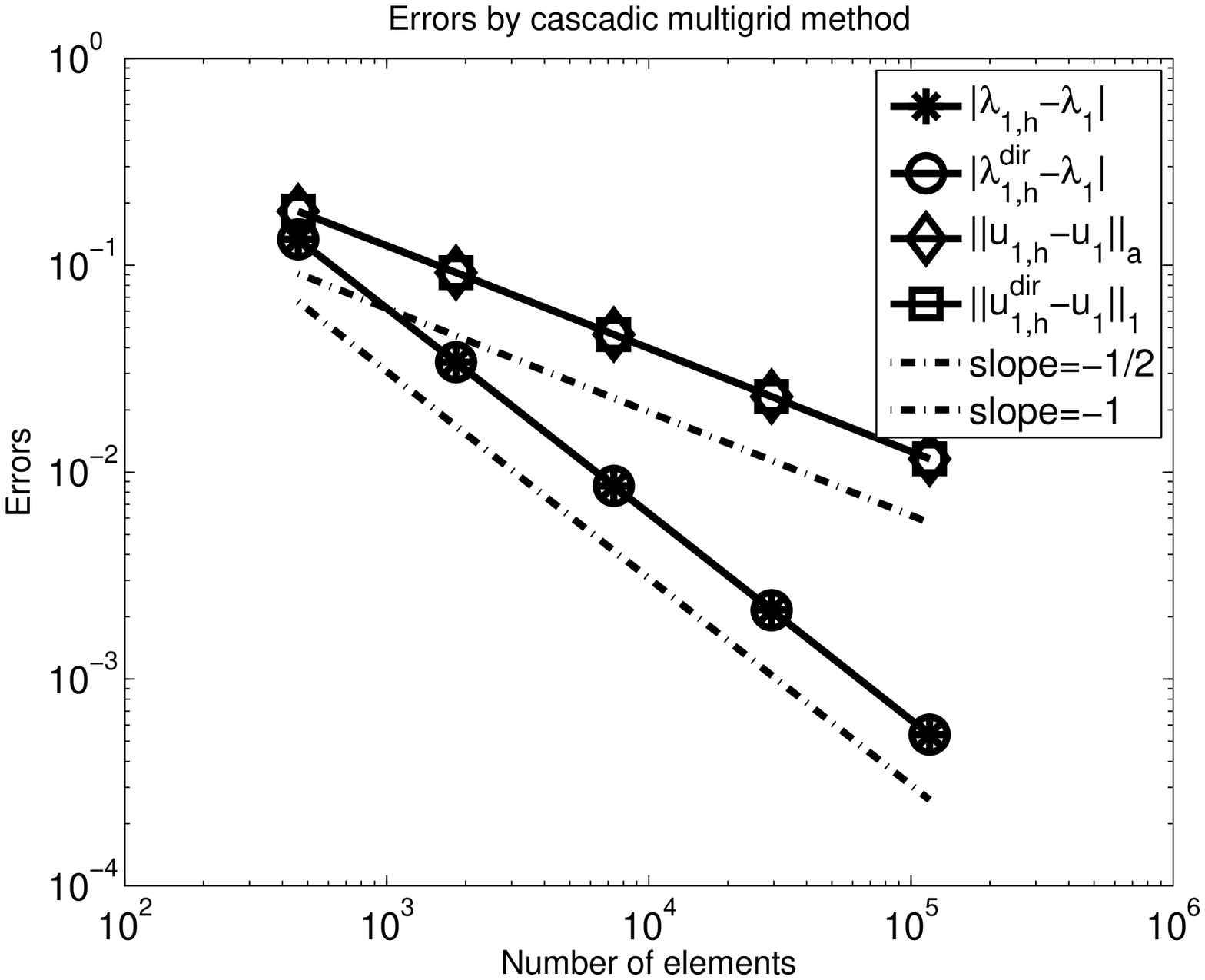}
\caption{\small\texttt The errors of the multigrid
algorithm for the first eigenvalue $2\pi^2$ and the corresponding eigenfunction,
where $u_h$ and $\lambda_h$ denote the eigenfunction and eigenvalue approximation by Algorithm \ref{Cascadic_MCS_Alg}, and $u_h^{\rm dir}$ and $\lambda_h^{\rm dir}$ denote the eigenfunction
 and eigenvalue approximation by direct eigenvalue solving (The left figure
 corresponds to the left mesh in Figure \ref{Initial_Mesh} and
the right figure corresponds to the right mesh in Figure \ref{Initial_Mesh}) }
\label{numerical_multi_grid_2D}
\end{figure}

From Figure \ref{numerical_multi_grid_2D},
we find the cascadic multigrid scheme can obtain
the optimal error estimates as same as the direct eigenvalue solving method for the eigenvalue and the
corresponding eigenfunction approximations. Furthermore, Figure \ref{numerical_multi_grid_2D}
also shows the computational work of Algorithm \ref{Cascadic_MCS_Alg} can arrive the optimality.

We also check the convergence behavior for multi eigenvalue approximations with Algorithm
\ref{Cascadic_MCS_Alg}. Here the first six eigenvalues
$\lambda=2\pi^2$, $5\pi^2$, $5\pi^2$, $8\pi^2$, $10\pi^2$, $10\pi^2$
are investigated. We also adopt the meshes shown in Figure \ref{Initial_Mesh} as
the initial mesh and the corresponding numerical results are shown
in Figure \ref{numerical_multi_grid_2D_6}.
Figure \ref{numerical_multi_grid_2D_6} also exhibits the
optimal convergence and complexity of the cascadic multigrid scheme.
\begin{figure}[htb]
\centering
\includegraphics[width=6cm,height=5cm]{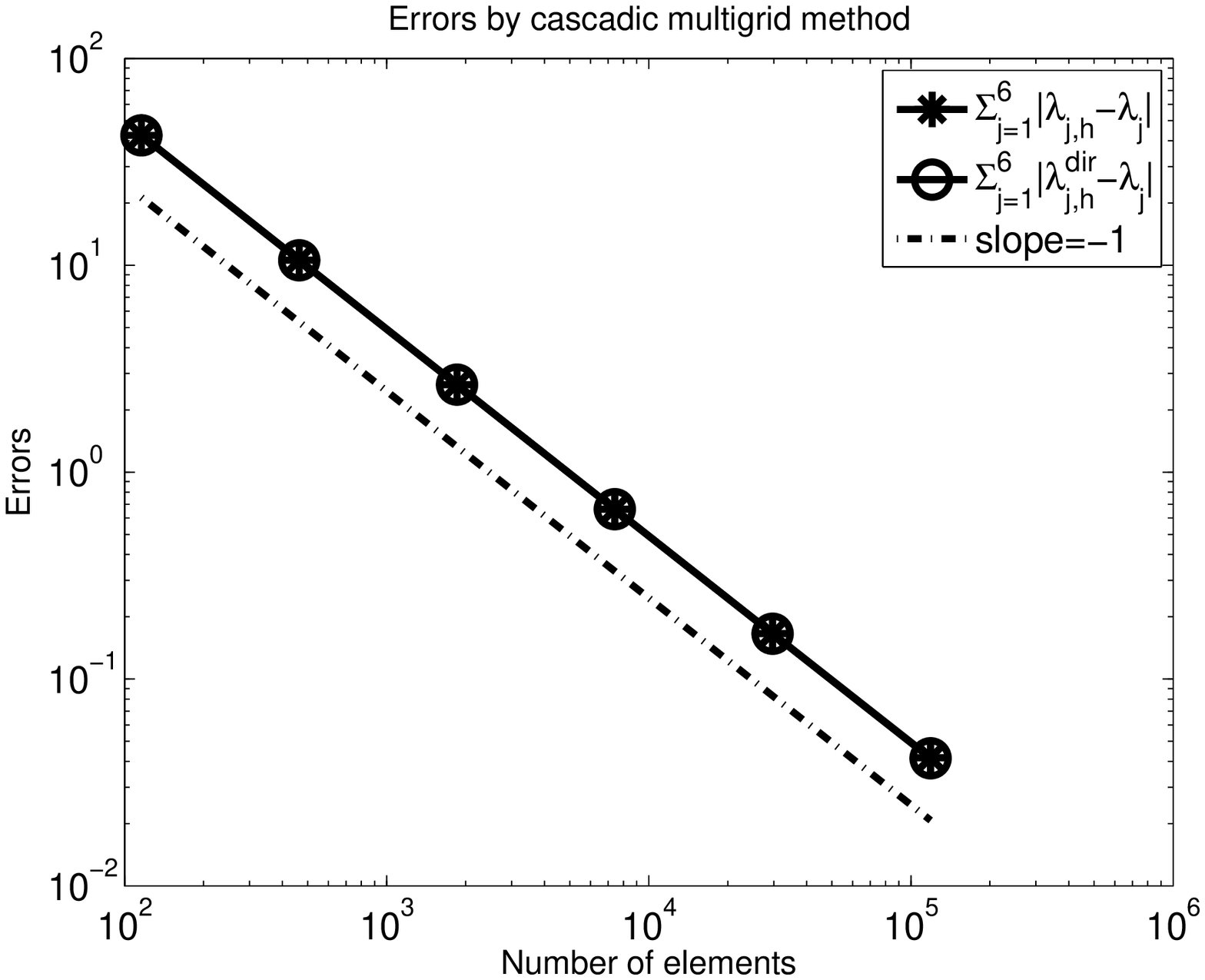}
\includegraphics[width=6cm,height=5cm]{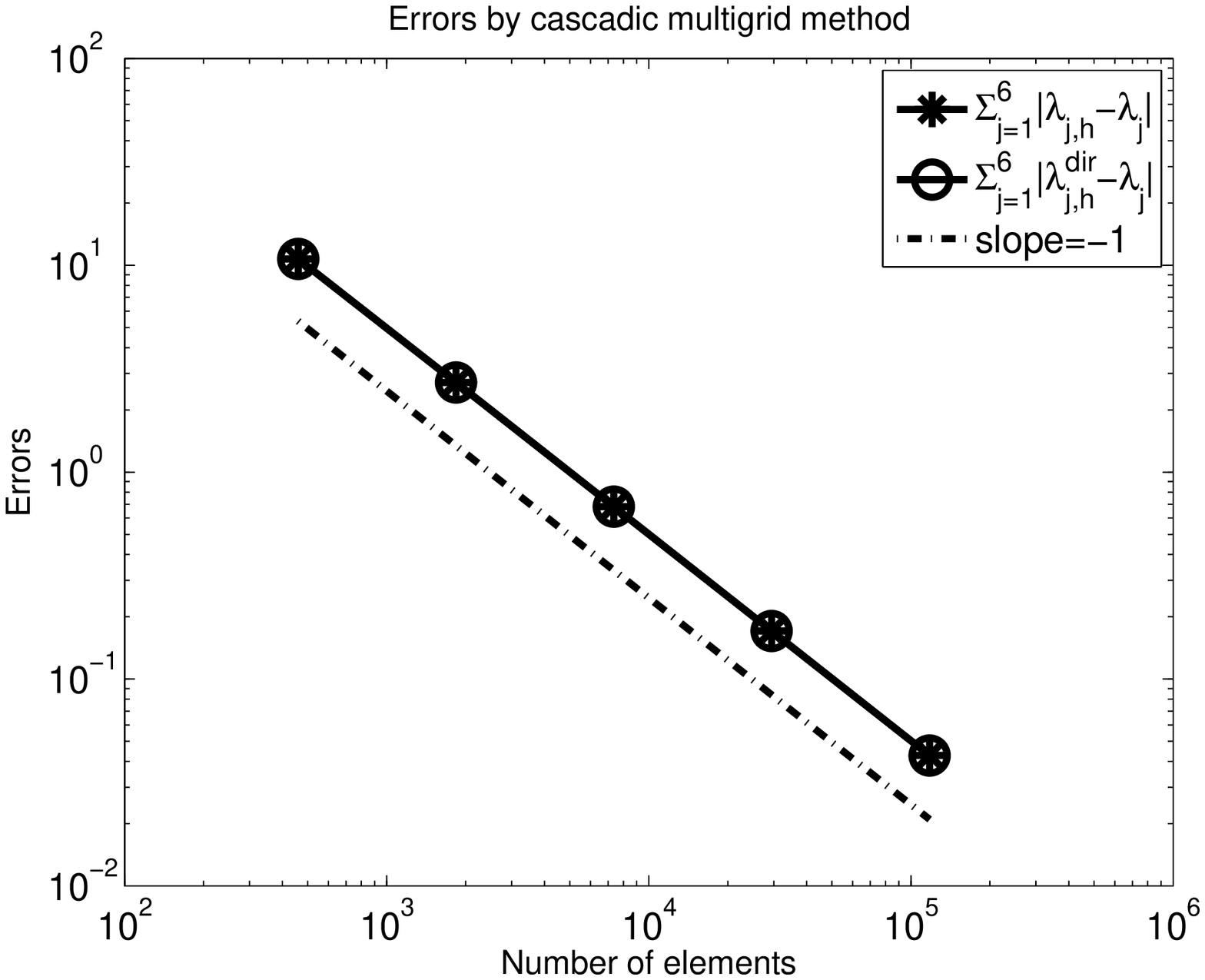}
\caption{\small\texttt The errors of the multigrid
algorithm for the first six eigenvalues on the unit square, where $u_h$ and $\lambda_h$ denote the eigenfunction and eigenvalue approximation by Algorithm \ref{Cascadic_MCS_Alg}, and $u_h^{\rm dir}$ and $\lambda_h^{\rm dir}$ denote the eigenfunction
 and eigenvalue approximation by direct eigenvalue solving
 (The left figure corresponds to the left mesh in Figure \ref{Initial_Mesh} and
the right figure corresponds to the right right mesh in Figure \ref{Initial_Mesh})}
\label{numerical_multi_grid_2D_6}
\end{figure}

\subsection{More general eigenvalue problem}
Here we give the numerical results of the cascadic multigrid
scheme for solving a more general eigenvalue problem on the unit square
domain $\Omega=(0, 1)\times (0, 1)$: Find $(\lambda,u)$ such that
\begin{equation}\label{Example_2}
\left\{
\begin{array}{rcl}
-\nabla\cdot\mathcal{A}\nabla u+\phi u&=&\lambda\rho u,\quad{\rm in}\ \Omega,\\
u&=&0,\quad\ \  {\rm on}\ \partial\Omega,\\
\int_{\Omega}\rho u^2d\Omega&=&1,
\end{array}
\right.
\end{equation}
where
\begin{equation*}
\mathcal{A}=\left (
\begin{array}{cc}
$$1+(x_1-\frac{1}{2})^2$$&$$(x_1-\frac{1}{2})(x_2-\frac{1}{2})$$\\
$$(x_1-\frac{1}{2})(x_2-\frac{1}{2})$$&$$1+(x_2-\frac{1}{2})^2$$
\end{array}
\right),
\end{equation*}
$\varphi=e^{(x_1-\frac{1}{2})(x_2-\frac{1}{2})}$ and
$\rho=1+(x_1-\frac{1}{2})(x_2-\frac{1}{2})$.

In this example, we also use two coarse meshes which are shown in Figure \ref{Initial_Mesh}
as the initial meshes to investigate the convergence behaviors.
Since the exact solution is not known, we choose an adequately accurate eigenvalue
approximations with the extrapolation method (see, e.g., \cite{LinLin}) as the exact
eigenvalues to measure errors.
Figure \ref{numerical_multi_grid_Exam_2} gives the corresponding
numerical results for the first six eigenvalue approximations.
Here we also compare the numerical results with the direct algorithm.
Figure \ref{numerical_multi_grid_Exam_2} also exhibits the optimality of the error and complexity for
Algorithm \ref{Cascadic_MCS_Alg}.
\begin{figure}[htb]
\centering
\includegraphics[width=6cm,height=5cm]{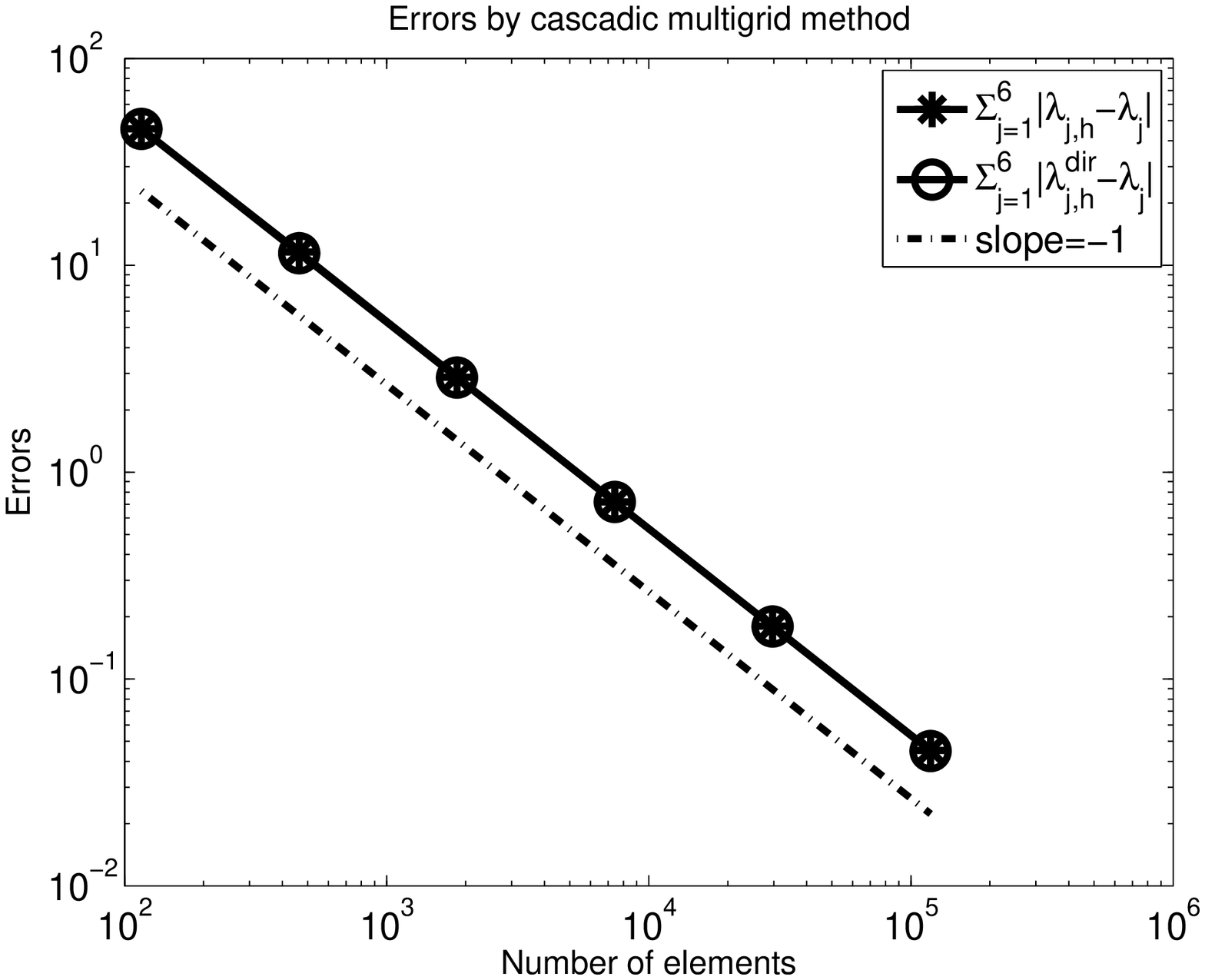}
\includegraphics[width=6cm,height=5cm]{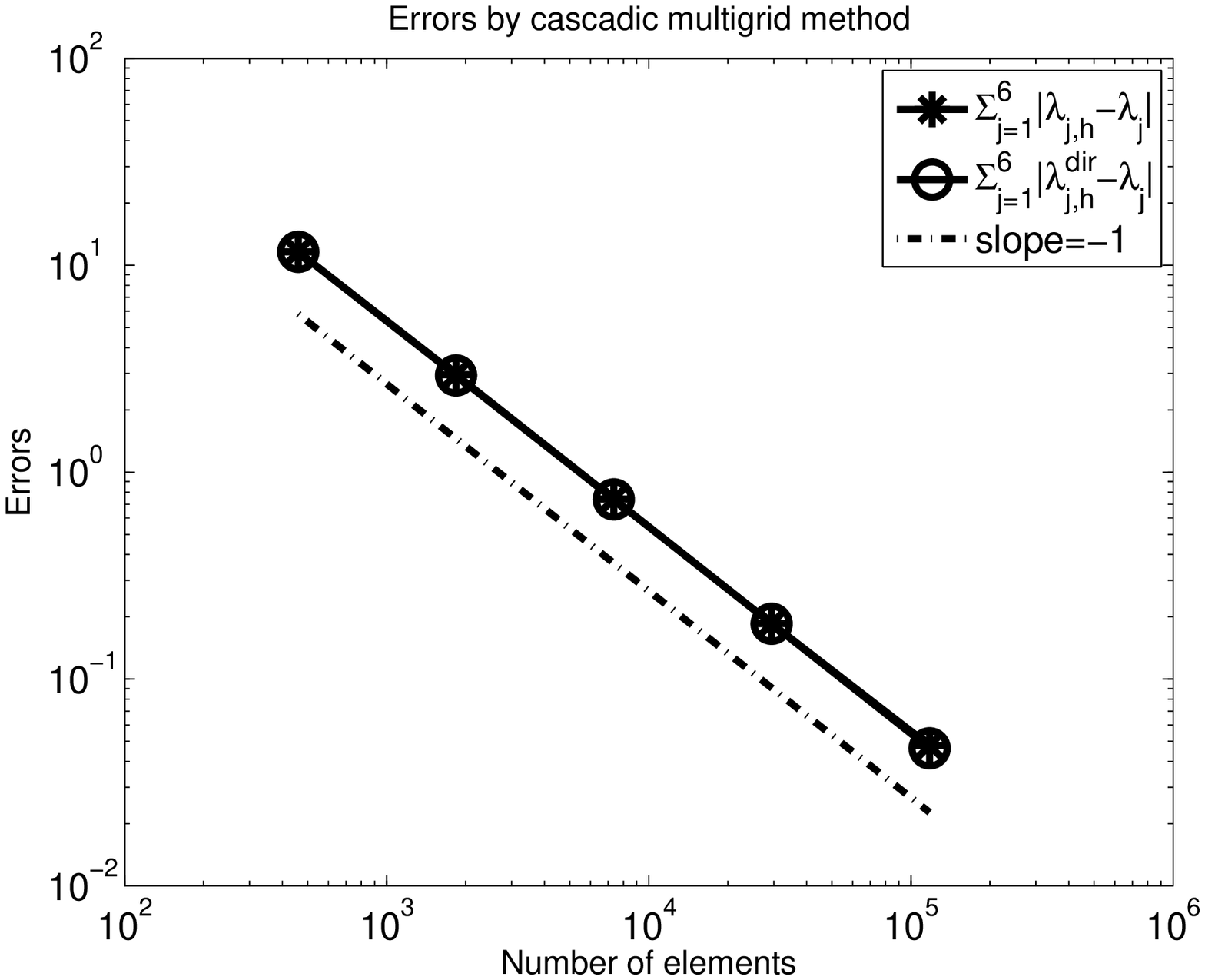}
\caption{\small\texttt The errors of the multigrid
algorithm for the first eigenvalue $3\pi^2$ and the corresponding eigenfunction,
where $u_h$ and $\lambda_h$ denote the eigenfunction and eigenvalue approximation by Algorithm \ref{Cascadic_MCS_Alg}, and $u_h^{\rm dir}$ and $\lambda_h^{\rm dir}$ denote the eigenfunction
 and eigenvalue approximation by direct eigenvalue solving
(The left figure corresponds to the left mesh in Figure \ref{Initial_Mesh} and
the right figure corresponds to the right right mesh in Figure \ref{Initial_Mesh})} \label{numerical_multi_grid_Exam_2}
\end{figure}

\section{Concluding remarks}

In this paper, we present a type of cascadic multigrid method for eigenvalue problems
based on the combination of the cascadic multigrid for boundary value problems and the multilevel correction
scheme for eigenvalue problems. The optimality of the computational efficiency
has been demonstrated by theoretical analysis and numerical examples.
As shown in the numerical examples, the cascadic multigrid method can also be used
to obtain the multiple eigenpair approximations of the eigenvalue problem(cf. \cite{Xie_IMA,Xie_JCP}).
Furthermore, the proposed
 cascadic multigrid method can be extended to more general eigenvalue problems.

\bibliographystyle{plain}

\end{document}